\documentclass[12pt]{amsart}
\usepackage[T1]{fontenc}
\usepackage[utf8]{inputenc}
\usepackage{color,amssymb,amsmath,mathrsfs,enumerate,esint}
\usepackage{a4wide}
\usepackage{hyperref}
\theoremstyle{plain}
\newtheorem{Theorem}{Theorem}[section]
\newtheorem{Lemma}[Theorem]{Lemma}

\newtheorem{Corollary}[Theorem]{Corollary}
\theoremstyle{definition}
\newtheorem{Example}[Theorem]{Example}
\newtheorem{Remark}[Theorem]{Remark}

\DeclareMathOperator*{\loc}{loc}

\DeclareMathOperator*{\esssup}{ess\,sup}
\title[Gagliardo-Nirenberg Inequality for r.i.BFS]{Gagliardo-Nirenberg Inequality for rearrangement-invariant Banach function spaces}
\author{A. Fiorenza, M.R. Formica, T. Roskovec and F. Soudsk\' y}
\address{Universit\'a di Napoli Federico II, Dipartimento di Architettura, via Monteoliveto, 3, 80134 - Napoli (Italy) and Consiglio Nazionale delle Ricerche, Istituto per le Applicazioni del Calcolo ``Mauro Picone", Sezione di Napoli, via Pietro Castellino, 111, 80131 - Napoli (Italy)}
\email{fiorenza@unina.it}
\address{Universit\'a di Napoli Parthenope, via Generale Parisi, 13, 80132 - Napoli (Italy)}
\email{formica@uniparthenope.it}
\address{Faculty of Economics, University of South Bohemia, Studentsk\' a 13, \v Cesk\' e Bud\v ejovice, Czech Republic and Faculty of Information Technology, Czech Technical University in Prague, Th\'{a}kurova 9, 160 00 Prague 6, Czech Republic}
\email{troskovec@ef.jcu.cz}
\address{Faculty of Economics, University of South Bohemia, Studentsk\' a 13, \v Cesk\' e Bud\v ejovice, Czech Republic}
\email{fsoudsky@ef.jcu.cz}
\date{December 7, 2018}
\begin{document}

\begin{abstract}
The classical Gagliardo--Nirenberg interpolation inequality is a well-known estimate which gives, in particular, an estimate for the Lebesgue norm of intermediate derivatives of functions in Sobolev spaces. We present an extension of this estimate into the scale of the general rearrangement--invariant Banach function spaces with the proof based on the Maz'ya's pointwise estimates. As corollaries, we present the Gagliardo--Nirenberg inequality for intermediate derivatives in the case of triples of Orlicz spaces and triples of Lorentz spaces. Finally, we promote the scaling argument to validate the optimality of the Gagliardo--Nirenberg inequality and show that the presented estimate in Orlicz scale is optimal.
\end{abstract}

\maketitle

\section{Introduction and Main Results}
The Sobolev--Gagliardo--Nirenberg inequality is one of the classical result obtained in many forms. We study the original form
$$\|\nabla^j u\|_{X}\lesssim \|\nabla^k u\|^{j/k}_{Y}\|u\|^{1-j/k}_{Z}.$$
Original results by Gagliardo \cite{Ga} and Nirenberg \cite{Ni} are focused on $X, Y, Z$ being the Lebesgue spaces. It is natural to extend these results to finer scales, considering, for instance, the Orlicz spaces (see \cite{KPP1, KPP2, KPP3, KPP4} by Ka{\l}amajska and Pietruska--Pa{\l}uba), the Lorentz spaces (see \cite{Mami} by Mart\'in and Milman, \cite{Kol2} by Kolyada and P\'erez L\'azaro, \cite{nguyendiaz} by Dao, D\'iaz and Nguyen, even if such papers deal with a slightly different setting) or the fractional Sobolev cases (see \cite{BM} by Brezis and Mironescu). There are also results where BMO or BV spaces are used \cite{MRR, KaMi, St, Led, nguyendiaz}.

There are more ways of proving such type of inequality. We should mention point-wise results by Bojarski and Haj{\l}asz \cite{BoHa} and Maz'ya and Shaposhnikova \cite{MS1, MS2, MS3} based on the maximal operator theory. Also, the theory of heat semigroups is applied by Mart\'in and Milman \cite{Mami}, Ledoux \cite{Led} or Kavian \cite{Kav}. Recently the theory of Besov and Lizorkin--Triebel spaces and wavelets are used by Brezis and Mironescu, Cohen et al., Ledoux, and Kolyada \cite{BM, CDDD, Led, Kol1}. We study and compare the original papers by Gagliardo and Nirenberg in \cite{FFRS}. In this paper we prove a generalisation using the  H\"older inequality in Banach function spaces over $\mathbb{R}^n$ endowed with the Lebesgue measure, therefore obtaining some of the previous results and some new results in an easy way. The introduced theorems are based only on basic Banach functional spaces properties, the pointwise estimate by Maz'ya and the optimal choice of spaces for H\"older inequality. Our calculations also contain the H\"older factorisation or the theory of multipliers, but we do not recall any advanced results of this theory.

Very similar methods to ours appear in \cite{van}, some Gagliardo--Nirenberg type inequalities in the framework of classical Sobolev spaces involving the localised Morrey norm are proven. The proof includes scaling arguments, the pointwise estimate by Maz'ya, and the maximal operator usage. 

Denoting $\varphi_X(t)$ the fundamental function of a given rearrangement-invariant Banach function space $X$, we may formulate a necessary condition for the Gagliardo--Nirenberg inequality to hold, which we get using a scaling argument. Scaling arguments are not new in the framework of the theory of Sobolev spaces (see, e.g. the standard arguments in Lectures 23, 30, 32 in the Tartar's book \cite{tartar}). However, the following result has the novelty to involve a general class of spaces and fundamental functions.

\begin{Theorem}[Scaling argument]\label{SA} Let $X,Y,Z$ be rearrangement invariant Banach function spaces over $\mathbb{R}^n$ endowed with the Lebesgue measure. Let $j,k\in\mathbb{N}$, $1\leq j<k$, and assume that 
\begin{equation}\label{GN}
\|\nabla^j u\|_{X}\lesssim \|\nabla^k u\|^{\frac{j}{k}}_{Y}\|u\|^{1-\frac{j}{k}}_{Z}
\end{equation}
holds for all $k$-times weakly differentiable functions $u$, with a constant independent of $u$.

Then the inequality
\begin{equation}\label{GNNC}
\varphi_{X}(t)\lesssim \varphi_{Y}^{j/k}(t)\varphi_{Z}^{1-j/k}(t)
\end{equation}
holds for all $t>0$ with a constant independent of $t$.
\end{Theorem}

In the following we shall use the symbol $$Y\stackrel{\textup{loc}}{\hookrightarrow}X,
$$ to denote the local embedding of the space $Y$ into space $X$ (for the precise definition see \eqref{locemb}); for the precise definition of the maximal rearrangement $u^{**}$, of $Y^{\frac{k}{j}}$ and of $Y^X$ see \eqref{maxf}, \eqref{power} and \eqref{FS} respectively. 
\begin{Theorem}[Gagliardo--Nirenberg inequality for r.i.BFS]\label{GBFSR}
If $j,k\in\mathbb{N}$, $1\leq j<k$, and if $X,Y$ are rearrangement invariant Banach function spaces over $\mathbb{R}^n$ such that 
$$
Y^{\frac{k}{j}}\stackrel{\textup{loc}}{\hookrightarrow} X\, ,
$$
then the estimate 
$$
\|\nabla^j u\|_{X}\lesssim\|(\nabla^k u)^{**}\|_{\raise -4pt \hbox{${}_Y$}}^{\frac{j}{k}}
\|u^{**}\|_{((Y^{\frac{k}{j}})^{X})^{1-\frac{j}{k}}}^{1-\frac{j}{k}}
$$
holds for all $k$-times weakly differentiable functions $u$ with a constant independent of $u$.
\end{Theorem}
\begin{Corollary}[Gagliardo--Nirenberg inequality for Lorentz spaces]\label{GNLOR}
Let $j,k\in\mathbb{N}$, $1\leq j<k$, $P,Q,R>1$, and $p,q,r\geq 1$ be numbers satisfying
\begin{equation}\label{assCor13}
\frac{1}{P}=\frac{\frac{j}{k}}{R}+\frac{1-\frac{j}{k}}{Q}, \quad \frac{1}{p}=\frac{\frac{j}{k}}{r}+\frac{1-\frac{j}{k}}{q}
\end{equation}
then the estimate
$$
\|\nabla^j u\|_{P,p}\lesssim \|\nabla^k u\|_{R,r}^{\frac{j}{k}}\|u\|_{Q,q}^{1-\frac{j}{k}}
$$
holds for all $k$-times weakly differentiable functions $u$ with a constant independent of $u$.
\end{Corollary}
We can easily prove the optimality of the choice for parameters $P, Q, R$, but the optimality of $p, q, r$ is still open, as we show later among the proofs.
\begin{Corollary}[Optimal Gagliardo--Nirenberg inequality for Orlicz spaces]\label{GNORL}
Let $j,k\in\mathbb{N}$, $1\leq  j<k$, let $A,B,C$ be a triple of Young functions, such that for all $t>0$ it holds
\begin{equation}\label{CFO}
B^{-1}(t)^{j/k}C^{-1}(t)^{1-j/k}\lesssim A^{-1}(t).
\end{equation}
Then the estimate
$$
\|\nabla^j u\|_{L^A}\lesssim \|(\nabla^k u)^{**}\|_{L^B}^{\frac{j}{k}}\|u^{**}\|^{1-\frac{j}{k}}_{L^C}
$$
holds for all $k$-times weakly differentiable functions $u$ with a constant independent of $u$.
Moreover, 
if the upper Boyd indices of $L^B$ and $L^C$ are smaller than $1$, then the inequality
\begin{equation}\label{GNOS}
\|\nabla^j u\|_{L^A}\lesssim\|\nabla^k u\|_{L^B}^{j/k}\|u\|_{L^C}^{1-j/k}
\end{equation}
holds for all $k$-times weakly differentiable functions $u$.

On the other hand, if condition \eqref{CFO} does not hold, the inequality \eqref{GNOS} fails to hold.
\end{Corollary}

The previously known results in the scale of Lorentz spaces are typically dealing with another type of Gagliardo--Nirenberg inequality, namely, $\|f\|_X\lesssim\|Df\|^\alpha_Y\|f\|^{1-\alpha}_Z$ (see for example \cite{Mami, Kol1, Kol2, nguyendiaz, figalli}). This type of inequalities are not directly covered in our setting, up to some cases using $L^{p,\infty}$ instead of the standard Lebesgue case in one of the terms. We observe that of course, one can get a family of such type of inequalities estimating the left-hand side of \eqref{GN} using Sobolev's theorem, but the consequent comparison with the existing literature would not be of interest in this paper.

The previously known results in the scale of Orlicz spaces are much richer. The most covering results are, up to our knowledge, by Ka{\l}amajska and Pietruska--Pa{\l}uba (see \cite{KPP1, KPP4}), and our results someway overlap with theirs. However, we remark that in \cite{KPP1} they assume the \sl ``\, Condition A'' \rm on the Young function $A$ (which imposes, in particular, the boundedness of $A(t)/t^2$ for $t$ small), hence, in the case $A$ power, the exponent must be greater or equal than $2$. On the other hand, in \cite{KPP4} (see Theorems 4.3 and 4.4 therein) the very general results, when considered in the case of the Lebesgue measure, are very similar to ours (as in our case, the results, even if obtained through different arguments, use inequalities involving the maximal operator). But we replace some technical assumptions on functions $A, B, C$ by much simpler conditions, in addition, we show the optimality of our result for an important scale of Orlicz spaces, and, finally, our inequalities hold for functions which are not necessarily compactly supported (as required in \cite{KPP4, KPP1}).

\section{Auxiliary results}

\subsection{Banach function spaces}
Let us first collect some auxiliary results involving general Banach function spaces. Proofs of these statements can be found in classical literature (the reader is referred, e.g. to \cite[Chapter 2]{BS}). 

Given a real-valued, (Lebesgue) measurable function $u$ defined on $\mathbb{R}^n$, define the \textit{non-increasing rearrangement} by
\begin{equation}\label{noninc}
u^{*}(t):=\inf\{s\in(0,\infty): |\{|u|>s\}|\leq t\}\, ,\qquad t>0
\end{equation}
(here we adopt the standard convention $\inf\emptyset=+\infty$) and the \textit{maximal function of $u^{*}$} by
\begin{equation}\label{maxf}
u^{**}(t):=\frac{1}{t}\int_{0}^{t}u^*(s)\textup{d}s\, ,\qquad t>0.
\end{equation}
We use symbol $u^{**}$ also for a (finite-dimensional) vector-valued function $u$: the definitions above are meaningful also in this case, taking into account that $|\cdot|$ in the definition of  $u^{*}$ stands for the norm in the euclidean space.

In this paper, we shall pay special attention to rearrangement-invariant Banach function spaces (shortly: r.i.BFS) over $\mathbb{R}^n$ endowed with the Lebesgue measure, which will be denoted by $X=X(\mathbb{R}^n)$. By the Luxemburg representation theorem (see \cite[Ch.2, Theorem 4.10]{BS}), the norm in such a space can be represented by the norm in a r.i.BFS $\overline{X}$ over $\mathbb{R}^+$ endowed with the Lebesgue measure, in a sense that 
$$
\|u\|_{X(\mathbb{R}^n)}=\|u^{*}\|_{\overline{X}(\mathbb{R}^+)}\, ;
$$
In the sequel, with abuse of notation, we will not make a distinction between the norms in $X(\mathbb{R}^n)$ and in $\overline{X}(\mathbb{R}^+)$, so that, for instance, we will feel free to write $\|u\|_{X}=\|u^{*}\|_{X}$. We recall that since two functions having the same non-increasing rearrangement have the same norm when $X$ is a r.i.BFS it makes sense to define the fundamental function through
$$
\varphi_X(t):=\|\chi_E\|_X\, , \qquad t>0\, ,
$$
where $E\subset \mathbb{R}^n$ is an arbitrary set of measure $t$. 

Given a Banach function space $X$, the functionals
\begin{equation}\label{power}
\|u\|_{X^{\alpha}}:=\left(\||u|^{\alpha}\|_{X}\right)^{\frac{1}{\alpha}}\, ,\qquad \alpha>0\, ,
\end{equation}
define a class of spaces which are often referred as a \textit{$\alpha$-convexification of $X$} (see for instance \cite{Loz,LT}). Note that if $\alpha>1$ then $\|\cdot\|_{X^{\alpha}}$ is a Banach function norm for any Banach function space $X$.

Let $X,Y$ be Banach function spaces. We say that $Y$ is \textit{locally embedded into} $X$, writing
$$
Y\stackrel{\textup{loc}}{\hookrightarrow}X
$$
if for every measurable set $E\subset \mathbb{R}^n$ with $|E|<\infty$ and every measurable function $u$, we have
\begin{equation}\label{locemb}
\|u\chi_E\|_X\leq C_E\|u\chi_E\|_Y,
\end{equation}
where the constant $C_E$ depends only on the set $E$. After this definition we can assert, for instance, that for any Banach function space $X(\mathbb{R}^n)$ it holds 
$$
L^\infty(\mathbb{R}^n)\stackrel{\textup{loc}}{\hookrightarrow}X(\mathbb{R}^n)\stackrel{\textup{loc}}{\hookrightarrow}L^1(\mathbb{R}^n)\, .
$$
In the following we shall use the standard definitions of the Lorentz spaces and the Orlicz spaces. For $1\leq p,q\leq \infty$ we define
$$
\|u\|_{p,q}:=\|t^{\frac{1}{p}-\frac{1}{q}}u^*(t)\|_q\, ;
$$
for $p=\infty$ or $q=\infty$ we consider $1/\infty=0$ in the exponent term. 

Function $A:[0,\infty)\to[0,\infty)$ is said to be a Young function if it is increasing, convex and satisfies 
$$\lim_{t\to\infty}\frac{A(t)}{t}=\infty,\, \lim_{t\to 0_+}\frac{A(t)}{t}=0.$$
Note that this definition entrains that any Young function $A$ is strictly positive on $(0,\infty)$ and $A(0)=0$ holds.
For Young function $A$, the modular $\rho_{A}$ is defined for measurable functions $f$ by
$$
\rho_{A}(f):=\int_{\mathbb{R}^n}A(|f(x)|)\textup{d}x
$$
and the corresponding Luxemburg norm is defined by
$$
\|u\|_{L^A}:=\inf\left\{\lambda>0:\rho_A\left(\frac{u}{\lambda}\right)\leq 1\right\}.
$$
Note that in case of Lorentz spaces the fundamental function is independent of the second exponent and
$$
\varphi_{L^{p,q}}(t)=t^{1/p}\, , \qquad t>0\, ,
$$
while for Orlicz spaces (see e.g. \cite[Ch.4, Lemma 8.17]{BS})
\begin{equation}\label{CHARO}
\varphi_{L^A}(t)=\frac{1}{A^{-1}(1/t)}\, , \qquad t>0\, .
\end{equation}

Making some elementary calculations one can easily obtain the convexification of these spaces.
\begin{Example}\label{ExampleFactorization} 
\begin{enumerate}[\upshape(i)]
\item[]
\item If $X=L^p$, then $X^{\alpha}=L^{\alpha p}.$
\item
If $A$ is a Young function and $X=L^{A}$ is the corresponding Orlicz space, then $X^{\alpha}=L^{B}$ where $B(t)=A(t^\alpha).$
\item
If $X=L^{p,q}$ is a Lorentz space, then $X^{\alpha}=L^{\alpha p,\alpha q}.$
\end{enumerate}
\end{Example}

Let us recall other useful well-known results. The first one is the \textit{Hardy-Littlewood-Polya principle} for r.i.BFS $X$ (see \cite[Ch.2, Corollary 4.7]{BS}): if there exists a constant $C>0$ such that it holds for any $t>0$
$$
\int_{0}^tu^*(s) \textup{d}s\leq C\int_{0}^{t}v^*(s)\textup{d}s\, ,
$$
then we have
$$
\|u\|_X\leq C\|v\|_X.
$$
Another classical result we need to recall is the \textit{Riesz--Herz equivalence} for the (uncentered, cubic) maximal operator. Setting, for scalar or vector-valued functions $u$ on $\mathbb{R}^n$,
$$
Mu(x):=\sup_{Q\ni x}\frac{1}{|Q|}\int_{Q}|u(y)|\textup{d}y\, , \qquad x\in \mathbb{R}^n\, ,
$$
where the supremum is taken over all cubes $Q\subset\mathbb{R}^n$ containing $x$, the Riesz--Herz equivalence states that
\begin{equation}
u^{**}(t)\approx (Mu)^*(t)\, , \qquad t>0\, ,
\end{equation}
in the sense that the two functions are majorized each other up to a multiplicative constant, not depending on $u$ but just on the dimension $n$ (see \cite[Ch.3, Theorem 3.8]{BS}, \cite{AKMP}).

Finally, let us recall that by the celebrated Lorentz-Shimogaki's theorem (see, e.g. \cite[Ch.3, Theorem 5.17]{BS}), the maximal operator is bounded in a general r.i.BFS if and only if its upper Boyd index is smaller than $1$. The general formula for the computation of the upper Boyd index of an Orlicz space from the generating Young function is, e.g. in \cite[Ch.4, Theorem 8.18]{BS}. Easier formulas, which can be used in most practical cases, are in \cite{FK1, FK2}.

\subsection{H\"older factorization of Banach function spaces}
In the following, we shall use the H\" older inequality in its most general form. Given a Banach function space $X$ we are looking for pairs of Banach function spaces $Y, Z$ such that
\begin{equation}\label{GHOL}
\|fg\|_X\lesssim\|f\|_{Y}\|g\|_Z.
\end{equation}
The classical H\" older inequality for Banach function spaces gives us estimates \eqref{GHOL} when $X=L^1$ (in this case, for any Banach function space $Y$, there exist the optimal partner space $Z$ -- i.e. the largest possible -- such that \eqref{GHOL} holds and this space is $Z:=Y'$, well-known as associated space). The goal is, for a given pair of spaces $X, Y$, to find -- if it exists -- the optimal space $Z$ satisfying the H\" older inequality \eqref{GHOL}. It is a kind of generalised version of associate space, which we may refer as space of \it H\" older multipliers. \rm This question was already discussed in many papers, for instance, \cite{ON, KoLM,schep}. 
A little different approach leading to similar results was taken in \cite{AF}. The following Lemma gives a formal definition and also the characterization of the pairs of Banach function spaces for which such a space exists. 
\begin{Lemma}[H\"older factorization]\label{Hold}
Let $X,Y$ be two Banach function spaces. The following conditions are equivalent: 
\begin{enumerate}[\upshape(i)]
\item
$$
Y\stackrel{\loc}{\hookrightarrow} X.
$$
\item
Functional $\|\cdot\|_{Z}$ given by
\begin{equation}\label{FS}
\left\|f\right\|_{Z}:=\displaystyle{\sup_{\|g\|_{Y}\leq 1}}\|fg\|_{X}
\end{equation}
is a Banach function norm.
\end{enumerate}
\end{Lemma}
\begin{proof}
At first we observe that (i) implies (ii): the reader can easily verify that the functional $\|\cdot\|_{Z}$ satisfies all the properties of the definition of the Banach function norm.\\

On the other hand, let (i) be violated. In that case, there exists a set $E$ of finite measure, such that for every $n\in\mathbb{N}$ there exists $f_{n}$ supported in $E$ satisfying
$$
\left\|f_{n}\right\|_{X}\geq n\left\|f_{n}\right\|_{Y}.
$$
By \cite[Theorem 1.8.]{BS}, there exists $g\in Y\setminus X$ with $\|g\|_{Y}=1$ supported in $E$ (note that the support in $E$ follows looking carefully at the proof of \cite[Theorem 1.8.]{BS}). We have
$$
\left\|\chi_{E}\right\|_{Z}\geq \|g\chi_{E}\|_{X}=\|g\|_{X}=\infty,
$$
therefore $\chi_{E}\notin Z$ and $Z$ is not a Banach function space.
\end{proof}
We shall denote the Banach function space $Z$ defined in \eqref{FS} by $Y^X$. 
By the definition of $Y^X$ we see that for a pair $X, Y$ of Banach function spaces such that
$$
Y\stackrel{\loc}{\hookrightarrow}X,
$$
we have a generalised version of H\" older inequality
\begin{equation}
\|fg\|_{X}\leq\|f\|_{Y^X}\|g\|_{Y}.
\end{equation}
Moreover, for fixed spaces $X, Y$ space $Y^X$ is the maximal (in the sense of embedding) Banach function space such that H\" older inequality  \eqref{GHOL} holds. For a detailed study of the space $Y^X$, a good reference is the paper \cite{MP} by Maligranda and Persson (see also, e.g. Pustylnik \cite{pus}).

To use the H\" older inequality, let us compute how does the space of these H\" older multipliers look like in the case of the Lorentz spaces and the Orlicz spaces. In the case of Orlicz spaces, the result is already known (see, e.g. \cite[Theorem 10.4 p.75]{maligrbook}), however, here we present it in the form that can be easily used later. 

\begin{Lemma}[H\" older inequality for Orlicz spaces]\label{ORLHOL}
Let $A,B$ be Young functions such that there exist $\delta>0$ and $K>0$ satisfying 
$$
A(t)\leq B(Kt)\, , \qquad  t>\delta\, ,
$$
and let $C$ be a Young function. 

Then the following conditions are equivalent.
\begin{enumerate}[\upshape(i)]
\item $$
L^C\hookrightarrow (L^B)^{L^A}.
$$
\item
There exists $K>1$ such that 
$$
A\left(\frac{st}{K}\right)\leq B(s)+C(t)\, , \qquad  s,t>0\, .
$$
\item
$$
C^{-1}(t)B^{-1}(t)\lesssim A^{-1}(t)\, , \qquad  t>0\, .
$$
\end{enumerate}

\end{Lemma}
\begin{Remark}
A variant of this Lemma can be found in a paper by O'Neil (see \cite[Theorem 6.5]{ON}), where the underlying set of the Orlicz spaces is bounded and the conditions on $A,B,C$ are
\begin{enumerate}
\item $$\limsup_{t\to\infty}C^{-1}(t)B^{-1}(t)/A^{-1}(t)<\infty.$$
\item There exists $K>0$ such that for all $s,t>0$ it holds$$A\left(\frac{st}{K}\right)\leq B(s)+C(t).$$
\item There exists $K'>0$ such that for $g,f$ measurable on $(0,1)$ it holds $$\|fg\|_{L^A}\leq K'\|f\|_{L^B}\|g\|_{L^C}.$$
\item For every $f\in L^B(0,1)$, $g\in L^C(0,1)$, $fg$ belongs to $L^A(0,1)$.
\end{enumerate}

Note that the condition (1) is equivalent to (iii) if one takes into account the boundedness of the underlying measure space. Condition (2) corresponds to (ii), while conditions (3) and (4) correspond to (i).

A very similar lemma can also be found in \cite[Theorem A.1, Lemma A.2]{Hogan}, where the authors use very minimal assumptions to get the equivalence of inequalities. The question of optimal H\"older inequality was initially studied by And\^o \cite{An}, he proved the equivalence of (i) and (ii): in his case (i) is just with equality, and in (ii) a supremum appears, because only the optimality is considered.

Despite the possibility of recalling modified versions of the proofs of the Theorems mentioned above, we present the complete proof for the convenience of the reader in Subsection \ref{sectionKnownResults}.
\end{Remark}



\begin{Lemma}[Saturated H\" older inequality for Lorentz spaces]\label{LemmaLorentz}
Let $P,p,Q,q,R,r$ be numbers for which $1\leq P,Q,R<\infty$ and $1\leq p,q,r$ and
$$
\frac{1}{P}=\frac{1}{R}+\frac{1}{Q}\quad \textup{and}\quad\frac{1}{p}=\frac{1}{r}+\frac{1}{q}. 
$$
Then
$$
L^{Q,q}=(L^{R,r})^{L^{P,p}}
$$
holds.
\end{Lemma}

\begin{Remark}
The version of H\" older inequality, implicitly included in Lemma \ref{LemmaLorentz}, is well known in the same assumptions, see, e.g. \cite[Theorem 4.5 p.271]{hunt}.
\end{Remark}

\subsection{Proof of the main results}
\begin{proof}[Proof of Theorem \ref{SA}]
In order to show inequality \eqref{GNNC}, we test inequality \eqref{GN} in a suitable way. Consider the function $u$ given by 
$$
u(x)=\left\{
\begin{array}{ll}
2-|x|^k & \textup{for }|x|<1\\
(2-|x|)^{k} & \textup{for }1\leq|x|\leq 2\\
0 & \textup{elsewhere}
\end{array}
\right.\, , \qquad  x\in \mathbb{R}^n \, .
$$ 
Define a dilation operator by
$$
T_{s}u(x):=u(sx)\, , \qquad  x\in \mathbb{R}^n \, , \quad s>0\, ,
$$
and set
$$
v:=T_s u.
$$
If \eqref{GN} holds then there is a constant independent of $s$ for which 
\begin{equation}\label{testedinequality}
\|\nabla^{j}v\|_{X}\leq C\|\nabla^{k}v\|_{Y}^{j/k}\|v\|_{Z}^{1-j/k}
\end{equation}
holds. We have
$$
\nabla^{j}v=s^{j}T_{s}(\nabla^{j} u), \quad \nabla^{k}v=s^{k}T_{s}(\nabla^{k} u).
$$
Moreover, elementary calculations show
$$
|T_s(\nabla^{k}u)|=k!\chi_{B(0,2/s)}
$$
and
$$
|T_{s}(\nabla^j u)|=\frac{k!}{(k-j)!}\left(|sx|^{k-j}\chi_{B(0,1/s)}(x)+|2-|sx||^{k-j}\chi_{B(0,2/s)\setminus B(0,1/s)}(x)\right).
$$
Therefore we conclude
$$
\begin{aligned}
(T_s u)^{**}&\approx(\chi_{(0,2/s)})^{**}\\
\left(\nabla^j (T_s u)\right)^{**}&\approx s^{j}(\chi_{(0,2/s)})^{**}\\
\left(\nabla^k (T_s u)\right)^{**}&\approx s^{k}(\chi_{(0,2/s)})^{**},
\end{aligned}
$$
where the constants of $\approx$ are independent of $s$. The Hardy--Littlewood--Polya principle yields
$$
\begin{aligned}
\|T_s u\|_Z&\approx \varphi_{Z}(2/s)\\ 
\|\nabla^j T_s u\|_X&\approx s^j\varphi_{X}(2/s)\\
\|\nabla^k T_s u\|_Y&\approx s^k\varphi_{Y}(2/s).
\end{aligned}
$$
Now from \eqref{testedinequality} it follows that
$$
s^j\varphi_{X}(2/s)\lesssim \left(s^k\varphi_{Y}(2/s)\right)^{j/k}\left(\varphi_{Z}(2/s)\right)^{1-j/k} \, , \quad s>0
$$
holds with constant independent of $s$. The inequality \eqref{GNNC} follows.
\end{proof}
In the case of Lebesgue spaces the scaling argument gives us the precise relationship between the exponents introduced by Gagliardo and Nirenberg. In the case of Orlicz spaces ($X=L^A, Y=L^B, Z=L^C$) we can use \eqref{CHARO} to verify that \eqref{GNNC} gives the following necessary condition, which can be expressed in the terms of an inequality involving the corresponding Young functions
\begin{equation}\label{NCOS}
B^{-1}(t)^{j/k}C^{-1}(t)^{1-j/k}\lesssim A^{-1}(t) \, , \quad t>0\, .
\end{equation}
Moreover, the analogous argument in the case of Lorentz spaces $L^{P,p},L^{R,r},L^{Q,q}$ gives us the following relationship \eqref{NCLS} between $P,R,Q$. 
\begin{Remark}
Using the scaling argument we see that for Lorentz spaces, the equality
\begin{equation}\label{NCLS}
\frac{1}{P}=\frac{j/k}{R}+\frac{1-j/k}{Q}.
\end{equation}
is necessary to hold. However there could be still a space for improvement in exponents $p,q,r$ since these do not have any effect on the fundamental function (the reader may found such comment also at the end of the Lecture 23 in the Tartar's book \cite{tartar}). We leave open the question of possible improvements.
\end{Remark}

\begin{Theorem}[Maz'ya's point-wise estimate]
Let $j,k\in\mathbb{N}$, $j<k$ and let $u$ be $k$-times weakly differentiable function. Then
$$
|\nabla^j u(x)|\lesssim M(\nabla^k u)(x)^{\frac{j}{k}}Mu(x)^{1-\frac{j}{k}}\, , \quad x\in\mathbb{R}^n\, .
$$
\end{Theorem}
\begin{proof}
See \cite[Theorem 1]{MS1}, which has been proved for the maximal operator over centered balls and therefore, a fortiori, holds for the uncentered cubic maximal operator.
\end{proof}
\begin{proof}[Proof of Theorem \ref{GBFSR}] Use the Maz'ya's point-wise estimate followed by the H\"older factorization and the Riesz--Herz equivalence to obtain
$$
\begin{aligned}
\|\nabla^j u\|_X&\lesssim \|(M(\nabla^k u))^{\frac{j}{k}} (Mu)^{1-\frac{j}{k}}\|_X\\
&\leq \|M(\nabla^k u)^{\frac{j}{k}}\|_{Y^{\frac{k}{j}}}
\|(Mu)^{1-\frac{j}{k}}\|_{(Y^{\frac{k}{j}})^X}\\
&=\|M(\nabla^{k}u)\|_{\raise -4pt \hbox{${}_Y$}}^{\frac{j}{k}}\|Mu\|_{(Y^{\frac{k}{j}})^X)^{1-\frac{j}{k}}}^{1-\frac{j}{k}}\\
&\approx \|(\nabla^{k}u)^{**}\|_{\raise -4pt \hbox{${}_Y$}}^{\frac{j}{k}}
\|u^{**}\|_{((Y^{\frac{k}{j}})^X)^{1-\frac{j}{k}}}^{1-\frac{j}{k}}.
\end{aligned}
$$
\end{proof}

\begin{proof}[Proof of Corollary \ref{GNLOR}]
Since
$$
\frac{k}{j}R>P,
$$
it is
$$
\left(L^{R,r}\right)^{\frac{k}{j}}\stackrel{\textup{loc}}{\hookrightarrow} L^{P,p}\, ,
$$
and therefore we are allowed to replace $X$ by $L^{P,p}$ and $Y$ by $L^{R,r}$ in Theorem \ref{GBFSR}, so that
\begin{equation}\label{almostdone}
\|\nabla^j u\|_{P,p}\lesssim
\|(\nabla^{k}u)^{**}\|_{R,r}^{\frac{j}{k}}
\|u^{**}\|_{(((L^{R,r})^{\frac{k}{j}})^{L^{P,p}})^{1-\frac{j}{k}}}^{1-\frac{j}{k}}.
\end{equation}
Applying in turn Example \ref{ExampleFactorization}(iii), Lemma \ref{LemmaLorentz}, Example \ref{ExampleFactorization}(iii) again, and, finally, \eqref{assCor13}, we have
$$
\left(\left((L^{R,r})^{\frac{k}{j}}\right)^{L^{P,p}}\right)^{1-\frac{j}{k}}
=
\left((L^{\frac{Rk}{j},\frac{rk}{j}})^{L^{P,p}}\right)^{1-\frac{j}{k}}
=
\left(L^{\frac{PR}{R-\frac{j}{k}P},\frac{pr}{r-\frac{j}{k}p}}\right)^{1-\frac{j}{k}}
=
L^{\frac{PR\left(1-\frac{j}{k}\right)}{R-\frac{j}{k}P},\frac{pr\left(1-\frac{j}{k}\right)}{r-\frac{j}{k}p}}
=L^{Q,q}\, .
$$
We now conclude dropping formally the maximal functions of the rearrangements in both terms of the right hand side of \eqref{almostdone}, because when $R,Q>1$ the functionals $\|\cdot\|_{R,r}$, $\|\cdot\|_{Q,q}$ are equivalent to Banach function norms (namely, those ones obtained replacing $u^{*}$ by $u^{**}$; for the proof see e.g. \cite[Corollary 8.2.4]{FS}), and therefore 
$$
 \|u^{**}\|_{R,r}\approx \|u\|_{R,r}\,, \qquad \|u^{**}\|_{Q,q}\approx \|u\|_{Q,q}.$$
\end{proof}

\begin{proof}[Proof of Corollary \ref{GNORL}]
Assumption \eqref{CFO} lets us to apply (iii)$\Rightarrow$(i) in Lemma \ref{ORLHOL} where $B(t)$ is replaced by $B(t^{\frac kj})$ and $C(t)$ is replaced by $C(t^{\frac k{k-j}})$, so that
$$
\left(L^C\right)^{\frac k{k-j}}\hookrightarrow \left((L^B)^{\frac kj}\right)^{L^A}
$$
and therefore, arguing as in the proof of Theorem \ref{GBFSR}, we get
$$
\begin{aligned}
\|\nabla^j u\|_{L^A}&\lesssim\|((\nabla^k u)^{**})^{\frac jk}(u^{**})^{1-\frac jk}\|_{L^A}\lesssim \|(\nabla^k u)^{**}\|_{L^B}^{\frac jk}\|u^{**}\|_{L^C}^{1-\frac jk}.
\end{aligned}
$$
Moreover, if the upper Boyd indices of $L^B$ and $L^C$ are smaller than $1$, then, as in the previous proof, we can drop formally the maximal functions of the rearrangements and we get 
the inequality \eqref{GNOS}. It is clear that in such case the assumption \eqref{CFO}, that we just showed to be sufficient for \eqref{GNOS}, is optimal because it is also necessary (see \eqref{NCOS}).
\end{proof}

We remark that the essential properties of the Orlicz spaces (with upper Boyd indices smaller than $1$) used in the proof are the easy computations of the convexification and the factorisation. In principle, the same method can be used for any other class of spaces with these same features.

\subsection{Proof of Lemmata}\label{sectionKnownResults}

\begin{proof}[Proof of Lemma \ref{ORLHOL}]

We shall prove $\textup{(i)}\Rightarrow\textup{(iii)}\Rightarrow\textup{(ii)}\Rightarrow\textup{(i)}$ 
\\

Let us start with the proof of $\textup{(i)}\Rightarrow\textup{(iii)}$: we plug
$$
g=f=\chi_{(0,t^{-1})} \, , \quad t>0\, ,
$$
into the H\" older inequality
$$
\|fg\|_{L^A}\lesssim \|f\|_{L^B}\|g\|_{L^C}\, ,
$$
and we get immediately
$$
B^{-1}(t)C^{-1}(t)\lesssim A^{-1}(t) \, , \quad t>0\, .
$$

Now let us prove $\textup{(iii)}\Rightarrow \textup{(ii)}$.
Let $K>1$ be such that 
$$
B^{-1}(t)C^{-1}(t)\le K A^{-1}(t) \, , \quad t>0\, .
$$
We use just the monotonicity to obtain
$$
\begin{aligned}
st&=B^{-1}(B(s))C^{-1}(C(t))\\
&\leq (B^{-1}C^{-1})(\operatorname{max}\{B(s),C(t)\})\\
&\leq (B^{-1}C^{-1})(B(s)+C(t))\\
&\leq KA^{-1}(B(s)+C(t))  \, , \quad s,t>0\, .
\end{aligned}
$$ 
Dividing by $K$ and applying $A$ to both sides we get 
$\textup{(ii)}$.

Now let us finish by proving $\textup{(ii)}\Rightarrow\textup{(i)}$. Let $f\in L^B$ and $g\in L^C$. Set $\hat{A}(t):=A\left(\frac{t}{K}\right)$. Since
$$
\rho_{\hat{A}}\left(\frac{fg}{2\|f\|_{L^B}\|g\|_{L^C}}\right)\leq\frac12
\rho_{\hat{A}}\left(\frac{fg}{\|f\|_{L^B}\|g\|_{L^C}}\right)\leq\frac12\left(\rho_B\left(\frac{f}{\|f\|_B}\right)+\rho_C\left(\frac{g}{\|g\|_C}\right)\right)\leq 1\, ,
$$
we deduce
$$
\|fg\|_{L^A}= K\|fg\|_{L^{\hat{A}}}\leq 2K\|f\|_{L^B}\|g\|_{L^C}
$$
hence we get
$$
L^C\hookrightarrow\left(L^A\right)^{L^B}.
$$
\end{proof}





\begin{proof}[Proof of Lemma \ref{LemmaLorentz}]
Consider the case $p<\infty$. First note that
\begin{equation}\label{LOHOL}
\begin{aligned}
\|fg\|_{P,p}^{p}&=\int_{0}^{\infty}(fg)^{*}(s)^{p}s^{\frac{p}{P}-1}\textup{d}s\\
&\leq\int_{0}^{\infty}f^*(s)^p g^*(s)^p s^{\frac{p}{P}-1}\textup{d}s\\
&=\int_{0}^{\infty}f^*(s)^p s^{\frac{p}{R}-\frac{p}{r}} g^*(s)^p s^{\frac{p}{Q}-\frac{p}{q}}\textup{d}s\\
&\leq \|f^*(s)^p  s^{\frac{p}{R}-\frac{p}{r}}\|_{\frac{r}{p}}\|g^*(s)^p s^{\frac{p}{Q}-\frac{p}{q}}\|_{\frac{q}{p}}\\
&=\|f\|_{R,r}^p\|g\|_{Q,q}^p.
\end{aligned}
\end{equation}
On the other hand, if $p=\infty$, we must have $r=\infty$ and $q=\infty$. As before we get
$$\begin{aligned}
\|fg\|_{P,\infty}&=\esssup\limits_{s>0}\{s^{1/P}(fg(s))^*\}\\
&\leq\esssup\limits_{s>0}\{s^{1/R+1/Q}f^*(s)g^*(s)\}\\
&\leq\esssup\limits_{s>0}\{s^{1/R}f^*(s)\}\esssup\limits_{s>0}\{s^{1/Q}g^*(s)\}\\
&=\|f\|_{R,\infty}\|g\|_{Q,\infty}\, .
\end{aligned}
$$

This proves $L^{Q,q}\hookrightarrow (L^{R,r})^{L^{P,p}}$. To prove the opposite embedding, we have to prove that for each function $g$ there exists a function $f$ such that all the inequalities in \eqref{LOHOL} are sharp (up to constant). Let $f$ be an arbitrary measurable function. Since the space $\mathbb{R}^n$ equipped by Lebesgue measure is resonant (see \cite[Theorem 2.7]{BS}) one can choose $g$ equi-measurable with $g^*(s)=f^{*}(s)^{r/p-1}s^{r/R-1}$ such that the first inequality in \eqref{LOHOL} is saturated up to arbitrary small $\varepsilon$. Since $g$ is chosen in this special way, H\" older inequality in \eqref{LOHOL} is saturated. And since $\varepsilon$ can be chosen arbitrary small the saturation is complete.
\end{proof}

\noindent
\it Acknowledgments. \rm The second author has been partially supported by the Gruppo Nazionale per l'Analisi Matematica, la Probabilit\`a e le loro Applicazioni (GNAMPA) of the Istituto Nazionale di Alta Matematica (INdAM) and by Universit\`a di Napoli Parthenope through the project ``Sostegno alla Ricerca individuale'',
the third author was supported by GA\v{C}R 18-00960Y, and the fourth author was supported by EF-IGS2017-Soudsk\' y-IGS07P1.

The authors wish to thank Raffaella Servadei, Luigi D'Onofrio and Associazione di Fondazioni e di Casse di Risparmio Spa for the Young Investigator Training Program 2017 that supported the third author during the research.

A special thank is for the friend and colleague Stanislav Hencl, who suggested to study the Gagliardo--Nirenberg inequality in the Lorentz scale: the whole research (begun with the paper \cite{FFRS}) came following this 
original goal.

\bibliographystyle{plain}

\end{document}